\documentclass{article}
\usepackage{amsmath, amsthm, amssymb, mathtools}
\usepackage{enumitem}
\usepackage[utf8]{inputenc}
\usepackage{mathrsfs}
\usepackage{geometry}
\usepackage{booktabs}
\usepackage{multirow}
\usepackage{tikz} 
\usepackage{xcolor}
\usetikzlibrary{arrows.meta, positioning}

\newtheorem{theorem}{Theorem}[section]
\newtheorem{lemma}[theorem]{Lemma}
\newtheorem{proposition}[theorem]{Proposition}
\newtheorem{corollary}[theorem]{Corollary}

{\theoremstyle{definition}
\newtheorem{definition}[theorem]{Definition}
\newtheorem{remark}[theorem]{Remark}
\newtheorem{example}[theorem]{Example}

\title{Twisted group algebras of faithful split metacyclic groups $C_p \rtimes C_m$ over finite fields}
\author{Sanjit Bhowmick\thanks{ E-mail: \rm{sanjitbhowmick@rnd.iitg.ac.in}, Department of Electronics and Electrical Engineering, Indian Institute of Technology Guwahati,
		Assam, 781039, India } \and Javier de la Cruz \thanks{ E-mail: \rm{jdelacruz@uninorte.edu.co}, Departamento de Matem\'aticas y Estad\'istica, Universidad del Norte, Barranquilla, Colombia} \and Edgar Martínez-Moro \thanks{ E-mail: \rm{edgar.martinez@uva.es}, Institute of Mathematics, University of Valladolid, Valladolid, Spain\\ Partially supported by Grant  SAGACT-1
MCIN/AEI/ 10.13039/501100011033 y FEDER "Una manera de hacer Europa"
PID2022-138906NB-C21
2023/2027}}}
 \begin{document}

\maketitle

\begin{abstract}
    Let $\mathbb{F}_\ell$ be a finite field with $\ell$ elements and let $G = C_p \rtimes C_m$ be a faithful split metacyclic group. 
    In this paper, we develop a complete theory for the twisted group algebra $\mathbb{F}_\ell^\alpha G$. Using the Lyndon--Hochschild--Serre spectral sequence, we prove that the second cohomology group of $G$  is isomorphic to  $\mathbb{F}_\ell^\times/(\mathbb{F}_\ell^\times)^m$, and we show that all twisting occurs only on the $C_m$ factor. We determine the primitive central idempotents by analyzing the combined action of the Frobenius automorphism and the group action on the character group of $C_p$. Using crossed product theory and the structure of finite fields, we obtain the complete Wedderburn decomposition of $\mathbb{F}_\ell^\alpha G$ into matrix algebras over explicitly determined fields $\mathbb{F}_{\ell^{d_j}}$. Finally, the irreducible projective representations of $G$ over $\mathbb{F}_\ell$ are also determined.
\end{abstract}

\section{Introduction}

Twisted group algebras were introduced by Schur 
\cite{Schur1904,Schur1907,Schur1911} and provide relationships between projective representation theory 
and ring theory. Twisted group algebras are a generalization of group algebras, where the product
is still induced by the product of the group, but where the result is twisted by a
scalar, which is determined by a 2-cocycle.  The semisimple case is particularly important in many settings; see, for instance, \cite{MargolisSchnabel2020}.
In addition, several aspects of some twisted group algebras over finite fields $\mathbb{F}_\ell$ have been studied \cite{Duarte2025,Duarte2026}, and they have practical value. 
They have been applied to the study of error-correcting codes 
\cite{TwistedSkew,DeLaCruzWillems2021,Duarte2024, Fan2025, VedeneyDeundyak2020} and to the design of  McEliece cryptosystems
\cite{DeLaCruzVillanueva2024,MartinezMolina2023}.  Also there are some structural results of when a linear code can be realized as an ideal of a group algebra; as well as in \cite{Bernal2009} the authors provided a necessary and sufficient condition for a linear code to be realized as an ideal in a finite group algebra, in \cite{DeLaCruzWillems2021} a similar result was proven for ideals in twisted group algebras, and in \cite{Duarte2024} the authors extended that characterization to crossed products. 

In this paper, we study twisted group algebras associated with the semidirect product group $G = C_p \rtimes C_m$, where $m \mid p-1$. When $m=2$, our results recover Duarte's theorems on dihedral groups in  \cite{Duarte2026}. In particular, our Theorem \ref{thm:final-classification} reduces to \cite[Theorem 12]{Duarte2026} when specialized to $m=2$. However, this generalization does not merely capture the essential features of the dihedral case, but also introduces new phenomena arising from the more complex action on the structure of the resulting simple components. We will provide the complete Wedderburn decomposition in explicit form, with tables summarizing all possible cases in the analysis.

It is worth mentioning that
  Theorem \ref{thm:H2} in this paper shows that $H^2(G,\mathbb{F}_\ell^\times) \cong H^2(C_m,\mathbb{F}_\ell^\times)$. This fact means that all cohomological twists are concentrated on the cyclic $C_m$ factor. However, this does not mean that the $C_p$ factor is irrelevant to the structure of the algebra $\mathbb{F}_\ell^\alpha G$. The action of $C_m$ on $C_p$ induces a permutation action on the Frobenius orbits of characters, and the resulting crossed product structure over the extension fields $\mathbb{F}_{\ell^f}$ can be highly nontrivial. The cocycle $\alpha$ restricts trivially to $C_p$, but the cohomology class reappears in the inner crossed product via the action of the stabilizer $H \subseteq C_m$ on the field $\mathbb{F}_{\ell^f}$.

The paper is organized as follows. Throughout the paper, we will consider the group $G = C_p \rtimes C_m$, where $p$ is an odd prime and $m$ is a positive integer such that $m \mid p-1$. In Section~\ref{sec:preliminaries}, we provide some preliminaries on twisted group algebras, cohomology of groups, and some results on the decomposition of the group algebra $\mathbb{F} C_p$. Section~\ref{sec:action} is devoted to describing the action of the cyclic group $C_m$ on the algebra $\mathbb{F} C_p$. In Section~\ref{sec:clasi} we provide a complete reduction of $\mathbb{F}_\ell^\alpha G$ to simple components. Finally, in Section~\ref{sec:complete} we provide a complete Wedderburn decomposition for $\mathbb{F}_\ell^\alpha G$, and Section~\ref{sec:examples} gives some worked examples that illustrate the Wedderburn decomposition for concrete instances.

\section{Preliminaries}
\label{sec:preliminaries}

Let $G$ be a finite group and let $A$ be an abelian group. 
A function $\alpha : G \times G \longrightarrow A$ is called a \emph{2-cocycle} 
if, for all $x,y,z \in G$, it satisfies
$$\alpha(x,y)\alpha(xy,z) 
    = 
    \alpha(y,z)\alpha(x,yz)
    \quad \text{and} \quad 
    \alpha(x,1)=\alpha(1,x)=1.$$
The first condition is known as the cocycle identity, and the second 
ensures the normalization of the cocycle.
If $t : G \longrightarrow A$ is a map with $t(1)=1$, then the function 
$\Delta(t) : G \times G \longrightarrow A$ defined by
\[
\Delta(t)(x,y)=t(x)t(y)t(xy)^{-1}
\]
is called a \emph{2-coboundary}. 
We will denote by $Z^2(G,A)$  the set of all 2-cocycles and by
$B^2(G,A)$  the set of all 2-coboundaries. 
The \emph{ second cohomology group} of $G$ with values in $A$ is defined as
\[
H^2(G,A)=Z^2(G,A)/B^2(G,A).
\]
Given a cocycle $\alpha\in Z^2(G,A)$, we denote its cohomology class by $[\alpha]$. When the Abelian group is given by the multiplicative units of the complex numbers $\mathbb{C}^\times$, we will call 
the \emph{Schur multiplier} of $G$ to second cohomology group
$
M(G)=H^2(G,\mathbb{C}^\times)$.

Let $g$ a generator of  $C_n=\langle g \rangle$, the cyclic group of order $n$, 
and let $\lambda \in \mathbb{F}_\ell^\times$. 
The map $\alpha_\lambda : C_n \times C_n \longrightarrow \mathbb{F}_\ell^\times$  given by
$$\alpha_\lambda(g^i,g^j)
    =
    \begin{cases}
        1, & \text{if } i+j<n, \\
        \lambda, & \text{if } i+j \ge n.
    \end{cases}$$
is a 2-cocycle (see \cite{Jacobson1975}) and it represents the cohomology class corresponding to $\lambda$ in $H^2(C_n,\mathbb{F}_\ell^\times) \cong \mathbb{F}_\ell^\times/(\mathbb{F}_\ell^\times)^n$.

\begin{definition}
    Let $\alpha \in Z^2(G,\mathbb{F}_\ell^\times)$. 
The \emph{twisted group algebra} $\mathbb{F}_\ell^\alpha G$ is defined 
as the vector space over $\mathbb{F}_\ell$ with basis 
$\{u_g\}_{g \in G}$, where multiplication is given by
\begin{equation*}
    u_g u_h = \alpha(g,h)\, u_{gh}
    \quad \text{for all } g,h \in G.
\end{equation*}
and scalars commute with the basis elements, that is, 
$u_g d = d u_g$ for every $d \in \mathbb{F}_\ell$.
\end{definition}

\begin{definition}[The Group $G = C_p \rtimes_r C_m$]
\label{sec:group}
Let $p$ be an odd prime and $m$ a positive integer such that $m \mid p-1$. Fix $r \in (\mathbb{Z}/p\mathbb{Z})^\times$ of order $m$, we  define
$$G = C_p \rtimes_r C_m = \langle a, b \mid a^p = 1,\ b^m = 1,\ b a b^{-1} = a^r \rangle.$$\end{definition}
Note that the action of $C_m$ on $C_p$ is faithful, since $r \not\equiv 1 \pmod p$. If the action were trivial, then $G \cong C_p \times C_m$ would be abelian; in this case, we do not treat it here.


From now on, we denote by $\mathbb{F} = \mathbb{F}_\ell$ the finite field with $\ell \neq p$ elements, and by $A = \mathbb{F}^\times$ the multiplicative group of $\mathbb{F}$, viewed as a trivial $G$-module.



\begin{theorem}
\label{thm:H2}
Let $G = C_p \rtimes C_m$ with $m \mid (p-1)$, and let $\mathbb{F} = \mathbb{F}_\ell$ be a finite field with $\ell \neq p$. Then
\[
H^2(G, \mathbb{F}^\times) \cong H^2(C_m, \mathbb{F}^\times) \cong \mathbb{F}^\times / (\mathbb{F}^\times)^m.
\]
\end{theorem}

\begin{proof}
Consider the Lyndon--Hochschild--Serre spectral sequence associated with the extension
\[
1 \longrightarrow C_p \longrightarrow G \longrightarrow C_m \longrightarrow 1,
\]
with coefficients $A = \mathbb{F}^\times$ endowed with the trivial $G$-action, then
\[
E_2^{i,j} = H^i(C_m, H^j(C_p, A)).
\]
Since $|A| = \ell - 1$ and $\gcd(p, \ell - 1) = 1$, the group $A$ has no $p$-torsion and hence,
$
H^j(C_p, A) = 0 $ for all indexes $j \ge 1$
(see \cite[Corollary 6.4]{Brown1982}). Therefore, the only nonzero term with $i + j = 2$ is
\[
E_2^{2,0} = H^2(C_m, A).
\]
It follows that the spectral sequence collapses at the $E_2$-page, and the inflation map
\[
\mathrm{inf} : H^2(C_m, A) \longrightarrow H^2(G, A)
\]
is an isomorphism.
Finally, since the action is trivial, we have that
$
H^2(C_m, \mathbb{F}^\times) \cong \mathbb{F}^\times / (\mathbb{F}^\times)^m
$
(see \cite[Chapter III, Section 6]{Brown1982}).
\end{proof}

\begin{remark}
\label{rem:cocycle-restriction}
Since the inflation map $H^2(C_m, \mathbb{F}^\times) \longrightarrow H^2(G, \mathbb{F}^\times)$ is an isomorphism, we may always choose a representative 2-cocycle $\alpha$ such that 
$$
\alpha|_{C_p \times C_p} = 1 \quad \text{and} \quad \alpha(c_p, c_m) = 1
$$
for all $c_p \in C_p$ and $c_m \in C_m$. We will use this normalization throughout the paper.
\end{remark}



Let $\widehat{C_p} = \operatorname{Hom}(C_p, \overline{\mathbb{F}}^\times)$ denote the character group of $C_p$. 
The Frobenius automorphism $\varphi: x \mapsto x^\ell$ acts on $\widehat{C_p}$ via
$$
(\varphi \cdot \chi)(a) = \chi(a)^\ell = \chi(a^\ell), \quad \text{for all } a \in C_p.
$$
Since $\chi$ takes values in roots of unity, this action simply permutes the characters.


\begin{definition}
\label{def:frobenius-orbit}
The \emph{Frobenius orbit} of a character $\chi \in \widehat{C_p}$ is the set
$$
\{\varphi^k \cdot \chi \mid k \ge 0\},
$$
where $\varphi$ denotes the Frobenius automorphism. 
The \emph{size} of the orbit is the order of $\ell$ modulo the order of $\chi$, i.e., the smallest positive integer $f$ such that 
$$
\ell^f \equiv 1 \pmod{\operatorname{ord}(\chi)}.
$$
\end{definition}



\begin{proposition}
\label{prop:group-algebra-decomp}
The group algebra $\mathbb{F} C_p$ decomposes as
$$
\mathbb{F} C_p \cong \bigoplus_{\mathcal{F}} \mathbb{F}(\chi_{\mathcal{F}}),
$$
where the sum runs over the Frobenius orbits $\mathcal{F}$ of nontrivial characters, and $\mathbb{F}(\chi_{\mathcal{F}})$ denotes the field extension of $\mathbb{F}$ generated by the values of any character in the orbit. For the trivial character, we obtain a copy of $\mathbb{F}$.
\end{proposition}

\begin{proof}
This follows from the fact that 
$$
\mathbb{F} C_p \cong \mathbb{F}[X]/\langle X^p - 1 \rangle
$$
and the factorization of $X^p - 1$ into cyclotomic polynomials over $\mathbb{F}$. 
The primitive central idempotents of $\mathbb{F} C_p$ correspond to Galois orbits of irreducible characters over $\mathbb{F}$ (see, e.g., \cite{Serre1977}).
\end{proof}


If $f$ is the order of $\ell$ modulo $p$, i.e., the smallest positive integer such that $p \mid \ell^f - 1$, then we have:
\begin{enumerate}
    \item The trivial character gives a 1-dimensional component $\mathbb{F}$.
    \item The nontrivial characters are partitioned into Frobenius orbits of size $f$.
    \item The number of such orbits is $(p-1)/f$.
\end{enumerate}
For each orbit, the corresponding simple component is the field $\mathbb{F}_{\ell^f}$. This yields the following result.

\begin{corollary} 
\label{cor:group-algebra} Whith the group $C_p$ defiened as above,
$$\mathbb{F} C_p \cong \mathbb{F} \oplus (\mathbb{F}_{\ell^f})^{(p-1)/f}.$$
\end{corollary}

\section{Action of $C_m$ on  $\mathbb{F} C_p$}
\label{sec:action}

The group $C_m = \langle b \rangle$ acts on the algebra $\mathbb{F} C_p$ via $\sigma_b(u_a) = u_{a^r}$. Note that if one checks this action at the character level,  it corresponds to $\chi \mapsto \chi^r$, where $\chi^r(a) = \chi(a)^r$.

\begin{lemma}
\label{lem:action-commutes}
The action of $C_m$ commutes with the Frobenius action. Henceforth, $C_m$ permutes the Frobenius orbits of characters.
\end{lemma}

\begin{proof}
For any character $\chi$ and any $k$, we compute
\[
\varphi(\chi^r)(a) = (\chi^r(a))^\ell = (\chi(a)^r)^\ell = \chi(a)^{r\ell} = (\chi(a)^\ell)^r = (\varphi(\chi)(a))^r = (\varphi(\chi))^r(a).
\]
Thus $\varphi \circ \sigma_b = \sigma_b \circ \varphi$ on characters.
\end{proof}

Thus, we have a well-defined action of the group $C_m$ on the set of Frobenius orbits of nontrivial characters. Let $\mathcal{O}_1, \dots, \mathcal{O}_s$ be the orbits of this action, and let $t_j = |\mathcal{O}_j|$ be the orbit sizes.
For each $C_m$-orbit $\mathcal{O}_j$ of Frobenius orbits, let $e_j$ be the sum of the primitive idempotents corresponding to all Frobenius orbits in $\mathcal{O}_j$, where $j=1,\ldots, s$.

\begin{lemma}
\label{lem:central-idempotents} With the notation above, 
each $e_j$ is a central idempotent in $\mathbb{F}^\alpha G$ , where $j=1,\ldots, s$.
\end{lemma}

\begin{proof}
Since $\sigma_b$ permutes the Frobenius orbits in $\mathcal{O}_j$, we have $\sigma_b(e_j) = e_j$. Thus $e_j$ commutes with $u_b$. Since $\alpha$ is trivial on $C_p \times C_p$, the subalgebra $\mathbb{F} C_p$ is commutative and centralizes itself. Therefore $e_j$ commutes with all of $\mathbb{F}^\alpha G$.
\end{proof}


Let $\mathcal{O}$ be a $C_m$-orbit of Frobenius orbits of nontrivial characters. We will fix the following notation:
\begin{enumerate}
    \item Each Frobenius orbit has size $f = \operatorname{ord}_p(\ell)$.
    \item The $C_m$-orbit $\mathcal{O}$ contains $t$ such Frobenius orbits.
    \item Let $K = \mathbb{F}_{\ell^f}$ be the field corresponding to each Frobenius orbit.
    \item Let $H \subseteq C_m$ be the stabilizer of a chosen Frobenius orbit in $\mathcal{O}$, so $|H| = h = m/t$.
\end{enumerate}

\begin{definition} Let $R$ be a ring and $G$ a finite group. Suppose $G$ acts on $R$ by
ring automorphisms via a homomorphism
$ 
\sigma : G \longrightarrow \operatorname{Aut}(R).
$ and
let $\alpha : G \times G \longrightarrow R^\times$ be a normalized $2$-cocycle.
The \emph{crossed product algebra} $(R,G,\sigma,\alpha)$ is the
$R$-module
\[
(R,G,\sigma,\alpha)=\bigoplus_{g\in G} R u_g
\]
with multiplication determined by
\begin{align*}
u_g r &= \sigma_g(r)\, u_g, \qquad r\in R, \\
u_g u_h &= \alpha(g,h)\, u_{gh}, \qquad g,h\in G.
\end{align*}
Here $\{u_g \mid g\in G\}$ are formal basis elements indexed by $G$.
The cocycle identity for $\alpha$ guarantees associativity.
\end{definition}

\begin{remark}
In this paper, we encounter crossed products of the following types.
\begin{itemize}

\item $(R, C_m, \sigma, \alpha_\lambda)$ denotes the crossed product of
the cyclic group $C_m=\langle b\rangle$ acting on the algebra $R$
via $\sigma : C_m \longrightarrow \operatorname{Aut}(R)$ with cocycle
$\alpha_\lambda : C_m \times C_m \longrightarrow \mathbb{F}^\times$.
Its multiplication satisfies
\[
u_b r = \sigma_b(r)u_b, \qquad
u_{b^i}u_{b^j}=\alpha_\lambda(b^i,b^j)u_{b^{i+j}}.
\]

\item $(K,H,\sigma|_H,\alpha|_H)$ denotes the crossed product obtained
by restricting the action and the cocycle to the stabilizer subgroup
$H \subseteq C_m$. Thus
\[
(K,H,\sigma|_H,\alpha|_H)
=\bigoplus_{h\in H} K u_h,
\]
with multiplication
\[
u_h k=\sigma_h(k)u_h, \qquad
u_hu_{h'}=\alpha(h,h')u_{hh'},
\]
for $k\in K$ and $h,h'\in H$.

\end{itemize}
\end{remark}
Now,
by the orbit-stabilizer theorem, $t$ must divide $m$, and $h = m/t$ is a positive integer and
therefore the central idempotent $e_{\mathcal{O}}$ gives a component
$$A_{\mathcal{O}} = \mathbb{F}^\alpha G e_{\mathcal{O}} \cong (R, C_m, \sigma, \alpha_\lambda),$$
where $R = \mathbb{F} C_p e_{\mathcal{O}}$ is a commutative semisimple algebra.
Since $\mathcal{O}$ contains $t$ Frobenius orbits, each contributing a copy of $K$, we have
$$R \cong \bigoplus_{i=1}^t K.$$

The group $C_m$ acts on $R$ by permuting these $t$ copies transitively, with stabilizer group $H$ in each copy.
The stabilizer $H$ acts on the chosen copy of $K$ by $\mathbb{F}$-algebra automorphisms. Since $H$ is cyclic (as a subgroup of $C_m$), this action corresponds to a subgroup of the Galois group $\operatorname{Gal}(K/\mathbb{F}) \cong C_f$. Since $\sigma_b^t$ restricts to an automorphism of $K$ over $\mathbb{F}_\ell$, and $\operatorname{Gal}(K/\mathbb{F}_\ell) \cong C_f$ is generated by the Frobenius automorphism $\varphi(\zeta)=\zeta^\ell$, there exists a unique $k \pmod{f}$ such that $$\sigma_b^t|_K = \varphi^k,$$ i.e., $r^t \equiv \ell^k \pmod{p}$ where $\varphi$ is the Frobenius automorphism. The order $s$ of $\sigma_b^t$ in $\operatorname{Gal}(K/\mathbb{F})$ is then the smallest positive integer such that $(\sigma_b^t)^s = \varphi^{ks}$ acts as the identity on $K$, i.e., such that $\ell^{ks} \equiv 1 \pmod{p}$.

\begin{lemma}
\label{lem:stabilizer-action}
The group $H$ acts on $K$  via the restriction of the automorphism $\sigma_b^t$ to $K$. 
This gives an embedding $$H \hookrightarrow \operatorname{Gal}(K/\mathbb{F}) \cong C_f.$$ Let $s$ be the order of 
$\sigma_b^t$ in $\operatorname{Gal}(K/\mathbb{F}) \cong C_f$. Then $s$ divides $h$,  $s$ divides $f$,  and  the fixed field $K^H$ has degree $[K^H:\mathbb{F}] = f/s$.
\end{lemma}

\begin{remark}
\label{rem:divisor-condition}
The conditions $s \mid h$ and $s \mid f$ are necessary consequences arising from group theory facts. For example, if $f = 4$ and $h = 3$, then $s$ must divide both $3$ and $4$, forcing $s = 1$ (trivial action). Similarly, if $f = 6$ and $h = 4$, then $s$ must divide $\gcd(4,6)=2$, so $s$ can only be $1$ or $2$. 
\end{remark}
\section{Clasification}\label{sec:clasi}



For the reader's convenience, we summarize the key notation used throughout this section:

\begin{itemize}
    \item $f = \operatorname{ord}_p(\ell)$: the size of each Frobenius orbit of nontrivial characters.
    \item $K = \mathbb{F}_{\ell^f}$: the field corresponding to each Frobenius orbit.
    \item $t = |\mathcal{O}|$: the size of a $C_m$-orbit of Frobenius orbits.
    \item $h = m/t$: the size of the stabilizer $H$ of a Frobenius orbit in $\mathcal{O}$.
    \item $s$: the order of $\sigma_b^t$ in $\operatorname{Gal}(K/\mathbb{F}_\ell) \cong C_f$ (so $s \mid \gcd(h, f)$).
    \item $d = f/s = [K^H:\mathbb{F}_\ell]$: the degree of the fixed field.
    \item $r = \sqrt{h s}$: a matrix size parameter (must be an integer).
    \item $M_t(K)$: the algebra of $t \times t$ matrices over $K$.
\end{itemize}

We recall some key structural results concerning crossed products with transitive permutation actions.


\begin{theorem}[\cite{Passman1989}]
\label{thm:crossed-permutation}
Let $R = \bigoplus_{i=1}^t K$, with $C_m$ acting transitively by permuting the summands, and let $H$ denote the stabilizer of a summand. Then
$$
(R, C_m, \sigma, \alpha) \cong M_t\big( (K, H, \sigma|_H, \alpha|_H) \big),
$$
where $(K, H, \sigma|_H, \alpha|_H)$ is the crossed product of $H$ over $K$ with the restricted action and cocycle.
\end{theorem}


 For the following discussion, we set
$$
B := (K, H, \sigma|_H, \alpha|_H),
$$
which is an algebra over $\mathbb{F}$ with
$$
\dim_{\mathbb{F}} B = |H| \cdot \dim_{\mathbb{F}} K = h \cdot f.
$$


\begin{proposition}
\label{prop:B-structure}
Let $B = (K, H, \sigma, \alpha)$ be a crossed product, where $K$ is a field and $H$ acts via field automorphisms. Then:
\begin{enumerate}
    \item The center of $B$ is $Z(B) = K^H$, the fixed field of $H$ acting on $K$.
    \item $B$ is a central simple algebra over $K^H$.
    \item If $K$ is finite, then 
    $ 
    B \cong M_r(K^H)
    $  
    for some positive integer $r$ satisfying
    $ 
    r^2 \cdot [K^H:\mathbb{F}] = h \cdot f.
    $ 
\end{enumerate}
\end{proposition}


\begin{proof}
For (1), see \cite[Theorem 4.4.1]{Passman1989}.  
For (2), it is well known that a crossed product is a central simple algebra over its fixed field.  
For (3), since the Brauer group of a finite field is trivial (see, e.g., \cite{Serre1977, Pierce1982}), every finite-dimensional central simple algebra over a finite field is isomorphic to a matrix algebra. The formula involving the dimension then follows by computing the dimensions over $\mathbb{F}$.  
\end{proof}

\begin{remark}
\label{rem:square-condition} Note that
the equality $r^2 = h\cdot s$ imposes a nontrivial constraint on the parameters that can be 
derived as follows. By Proposition~\ref{prop:B-structure}, $B \cong M_r(K^H)$ is a matrix algebra of size $r$ over $K^H$, so its dimension over $\mathbb{F}_\ell$ is
$$
\dim_{\mathbb{F}_\ell} B = r^2 \cdot [K^H:\mathbb{F}_\ell] = r^2 \cdot d.
$$
From the crossed product structure, we also have that
$$
\dim_{\mathbb{F}_\ell} B = |H| \cdot \dim_{\mathbb{F}_\ell} K = h\cdot f.
$$
Now, since $d = f/s$ by Lemma~\ref{lem:stabilizer-action}, equating both expressions gives
$$
r^2 \cdot \mathbb{F}rac{f}{s} = h f \quad \Longrightarrow \quad r^2 = h\cdot s.
$$
Thus, $hs$ must be a perfect square, and $r = \sqrt{h\cdot s}$ is the integer matrix size.  
This is not an additional assumption but a necessary consequence of the structure theory: the crossed product $B$ is central simple over $K^H$, and over a finite field every such algebra is a matrix algebra, forcing $r$ to be a positive integer. Consequently, only parameter combinations with $h\cdot s$ being a perfect square can occur in any valid group action.\\
This condition, together with $s \mid \gcd(h,f)$ from Lemma~\ref{lem:stabilizer-action}, provides strong constraints that eliminate many impossible parameter combinations and serve as useful consistency checks for the theory.
\end{remark}
\begin{example} Consider $m = 4$, $f = 6$, and the stabilizer action has order $s = 2$, then with $t = 1$ we have $h = m/t = 4$, giving $r = \sqrt{4 \cdot 2} = \sqrt{8}$, which is not an integer.  
Hence, such a configuration cannot arise from any actual group action, and the parameters $(m,f,s) = (4,6,2)$ are incompatible.
\end{example}


From Theorem~\ref{thm:crossed-permutation} and Proposition~\ref{prop:B-structure}, we have
$$A_{\mathcal{O}} \cong M_t(B), \quad Z(A_{\mathcal{O}}) = Z(B) = K^H.$$
Hence,
$$\dim_{\mathbb{F}} A_{\mathcal{O}} = t^2 \cdot \dim_{\mathbb{F}} B = t^2 \cdot h \cdot f = t^2 \cdot \mathbb{F}rac{m}{t} \cdot f = t\cdot m\cdot f,$$
which agrees with the expected dimension
$$
\dim_{\mathbb{F}} R \cdot m = (t\cdot f) \cdot m = t\cdot m\cdot f.
$$
The following result provides a classification of the component $A_{\mathcal{O}}$ corresponding to the central idempotent $e_{\mathcal{O}}$ in terms of the stabilizer $H$ and its action on the field $K$.


    
    
    

\begin{theorem}
\label{thm:final-classification}
Let $\mathcal{O}$ be a $C_m$-orbit of Frobenius orbits with parameters $f$, $t$, and stabilizer $H$ of size $h = m/t$. Let $s$ be the order of $\sigma_b^t$ in $\operatorname{Gal}(K/\mathbb{F}_\ell) \cong C_f$. Define
\begin{align*}
d &= [K^H:\mathbb{F}_\ell] = f/s, \\
r &= \sqrt{h \cdot s}.
\end{align*}
Then:
\begin{enumerate}
    \item \textbf{General case:} 
    $A_{\mathcal{O}} \cong M_{tr}(K^H)$, a matrix algebra over the field $K^H = \mathbb{F}_{\ell^{d}}$.
    
    \item \textbf{Special case $h = 1$:} 
    Then $t = m$, $H$ is trivial, $s = 1$, $d = f$, $r = 1$, and 
    $A_{\mathcal{O}} \cong M_m(K)$. 
    This occurs when $C_m$ acts transitively on the Frobenius orbits.
    
    \item \textbf{Special case $s = 1$:} 
    Then $H$ acts trivially on $K$, so $K^H = K$, $d = f$, $r = \sqrt{h}$, and 
    $A_{\mathcal{O}} \cong M_{t\sqrt{h}}(K)$.
    
    \item \textbf{Special case $t = 1$:} 
    Then $H = C_m$, so the full group stabilizes a single Frobenius orbit. In this case,
    $A_{\mathcal{O}} \cong M_r(K^H)$ with $r = \sqrt{m \cdot s}$, and the structure depends on whether the action of $C_m$ on $K$ is nontrivial ($s > 1$) or trivial ($s = 1$).
\end{enumerate}
In all cases, $A_{\mathcal{O}}$ is a simple algebra, and the decomposition of $\mathbb{F}^\alpha G$ into simple components is given by
$$
\mathbb{F}^\alpha G \cong \mathbb{F}^{\alpha_\lambda} C_m \oplus \bigoplus_{\mathcal{O}} A_{\mathcal{O}},
$$
where the sum runs over all $C_m$-orbits of nontrivial Frobenius orbits.
\end{theorem}

\begin{remark}
\label{rem:orbit-structure-analysis}
The decomposition in Theorem \ref{thm:final-classification} covers  all possible actions of $C_m$ on Frobenius orbits. To determine which case applies to a given group, one may follow the procedure below:

\begin{itemize}
    \item \textbf{Step 1: Compute $f = \operatorname{ord}_p(\ell)$}. This determines the size of each Frobenius orbit and the field $K = \mathbb{F}_{\ell^f}$.
   
      \item \textbf{Step 2: Determine the number of Frobenius orbits.} 
    This is given by $N = (p-1)/f$.
    
    \item \textbf{Step 3: Analyze the action of $C_m = \langle b \rangle$ on the set of Frobenius orbits}. 
    The element $b$ acts by sending a character $\chi$ to $\chi^r$, which permutes the Frobenius orbits. 
    Compute the orbit decomposition of this action. Let $\mathcal{O}_1, \ldots, \mathcal{O}_u$ be the $C_m$-orbits 
    of Frobenius orbits, with sizes $t_1, \ldots, t_u$. Note that $\sum_{j=1}^u t_j = N$.

    \item \textbf{Step 4: For each orbit $\mathcal{O}_j$ with size $t_j$, compute:}
    \begin{itemize}
        \item $h_j = m/t_j$, the size of the stabilizer $H_j$ of a Frobenius orbit in $\mathcal{O}_j$
        \item $s_j$, the order of $\sigma_b^{t_j}|_K$ in $\operatorname{Gal}(K/\mathbb{F}_\ell) \cong C_f$
        \item Verify the necessary conditions: $s_j \mid \gcd(h_j, f)$ (by Lemma \ref{lem:stabilizer-action})
        \item Verify that $h_j s_j$ is a perfect square (by Remark \ref{rem:square-condition})
        \item Compute $d_j = f/s_j$ and $r_j = \sqrt{h_j s_j}$
    \end{itemize}
    
    \item \textbf{Step 5: Identify which special case applies:}
    \begin{itemize}
        \item If $t_j = m$, then $h_j = 1$: this is the transitive case, giving $A_{\mathcal{O}_j} \cong M_m(K)$
        \item If $t_j = 1$, then $h_j = m$: this is the fixed point case, giving $A_{\mathcal{O}_j} \cong M_{r_j}(K^{H_j})$ with $r_j = \sqrt{m s_j}$
        \item If $1 < t_j < m$, then we have a nontrivial intermediate case with $1 < h_j < m$
    \end{itemize}
\end{itemize}
Note that the parameters must satisfy all consistency conditions; impossible combinations (such as those violating $s_j \mid \gcd(h_j, f)$ or the perfect square condition) simply cannot occur in any valid group action.
\end{remark}

\begin{corollary}
\label{cor:dimension-check}
For each $C_m$-orbit $\mathcal{O}_j$, the dimension of the corresponding simple component is given by
\[
\dim_{\mathbb{F}_\ell} A_{\mathcal{O}_j} = (t_j\cdot r_j)^2 \cdot d_j = t_j^2 \cdot h_j s_j \cdot \mathbb{F}rac{f}{s_j} = t_j^2\cdot h_j\cdot f = t_j\cdot m\cdot f.
\]
\end{corollary}
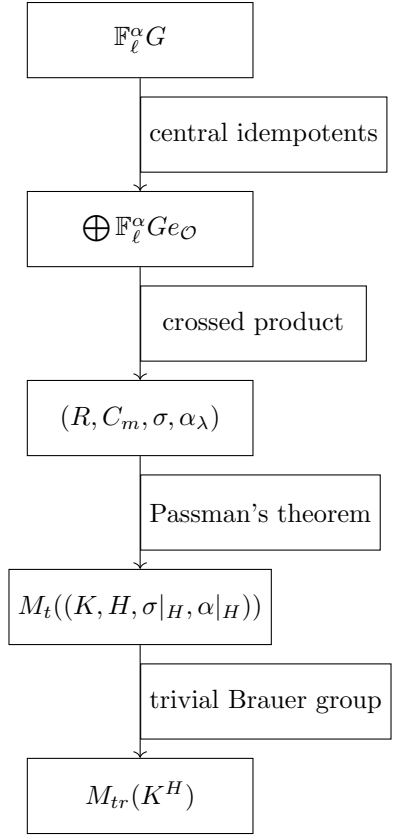
\begin{figure}[t]
\centering
\begin{tikzpicture}[node distance=1.5cm, auto, every node/.style={rectangle, draw, minimum width=3cm, minimum height=1cm, align=center}]
    \node (start) {$\mathbb{F}_\ell^\alpha G$};
    \node (idemp) [below=of start] {$\bigoplus \mathbb{F}_\ell^\alpha G e_{\mathcal{O}}$};
    \node (cross) [below=of idemp] {$(R, C_m, \sigma, \alpha_\lambda)$};
    \node (perm) [below=of cross] {$M_t((K, H, \sigma|_H, \alpha|_H))$};
    \node (final) [below=of perm] {$M_{tr}(K^H)$};
    
    \draw[->] (start) -- node[right] {central idempotents} (idemp);
    \draw[->] (idemp) -- node[right] {crossed product} (cross);
    \draw[->] (cross) -- node[right] {Passman's theorem} (perm);
    \draw[->] (perm) -- node[right] {trivial Brauer group} (final);
\end{tikzpicture}
\caption{Reduction of $\mathbb{F}_\ell^\alpha G$ to simple components}
\label{fig:flowchart}
\end{figure}
Note that summing over all orbits and adding the commutative component given by $\mathbb{F}^{\alpha_\lambda} C_m$ (of dimension $m$), it follows that 
\[
\dim_{\mathbb{F}_\ell} \mathbb{F}^\alpha G = m + \sum_{j=1}^u t_j\cdot m\cdot f = m\cdot \left(1 + f\sum_{j=1}^u t_j\right) = m\left(1 + f \cdot \mathbb{F}rac{p-1}{f}\right) = m\cdot p = |G|,
\]
confirming that the decomposition is dimension-preserving.
Figure~\ref{fig:flowchart}  illustrates the step-by-step reduction from the twisted group algebra to its simple components.

\section{Complete Wedderburn decomposition}
\label{sec:complete}

We now present the complete Wedderburn decomposition of $\mathbb{F}_\ell^\alpha G$ in explicit form. We begin by analyzing the component coming from the trivial character of $C_p$, then provide tables summarizing all cases. The component corresponding to the trivial character of $C_p$ is
$$A_0 = \mathbb{F}_\ell^\alpha G e_{\text{triv}} \cong \mathbb{F}_\ell^{\alpha_\lambda} C_m.$$
This is a twisted group algebra of the cyclic group $C_m$. Its structure depends on $\lambda \in \mathbb{F}_\ell^\times/(\mathbb{F}_\ell^\times)^m$ and the field $\mathbb{F}_\ell$.

\begin{proposition}
\label{prop:cyclic-twisted}
Let $\lambda \in \mathbb{F}_\ell^\times$ represent a cohomology class in $\mathbb{F}_\ell^\times/(\mathbb{F}_\ell^\times)^m$. Then
$$\mathbb{F}_\ell^{\alpha_\lambda} C_m \cong \bigoplus_{d|m} \mathbb{F}_{\ell^d}^{\oplus a_d(\lambda)},$$
where $a_d(\lambda)$ are determined as follows
\begin{enumerate}
    \item If $\lambda \in (\mathbb{F}_\ell^\times)^m$, then $\alpha_\lambda$ is cohomologous to the trivial cocycle, and $\mathbb{F}_\ell^{\alpha_\lambda} C_m \cong \mathbb{F}_\ell C_m \cong \bigoplus_{d|m} \mathbb{F}_{\ell^d}^{\oplus \varphi(d)}$, where $\varphi$ is Euler's totient function.
    \item Otherwise, let $s$ be the order of $\lambda$ in $\mathbb{F}_\ell^\times/(\mathbb{F}_\ell^\times)^m$. Then $\mathbb{F}_\ell^{\alpha_\lambda} C_m \cong \mathbb{F}_{\ell^{m/s}} \oplus \bigoplus_{d|m, d \nmid m/s} \mathbb{F}_{\ell^d}^{\oplus b_d}$ for appropriate multiplicities $b_d$ (see \cite{Karpilovsky1993} for the precise formulas).
\end{enumerate}
\end{proposition}
For simplicity in our tables, we denote the commutative component as $C(\lambda)$, with dimension $\dim C(\lambda) = m$.
For each $C_m$-orbit $\mathcal{O}$ of Frobenius orbits of nontrivial characters, we need to compute

\begin{itemize}
    \item $f = \operatorname{ord}_p(\ell)$
    \item $t = |\mathcal{O}|$: the size of the $C_m$-orbit of Frobenius orbits
    \item $h = m/t$: the size of the stabilizer
    \item $K = \mathbb{F}_{\ell^f}$
    \item $s$: the order of $\sigma_b^t$ in $\operatorname{Gal}(K/\mathbb{F}_\ell) \cong C_f$ (note that $s \mid \gcd(h, f)$ by Lemma~\ref{lem:stabilizer-action})
    \item $d = f/s$: the degree of the fixed field
    \item $r = \sqrt{h \cdot s}$: a positive integer (see Remark~\ref{rem:square-condition})
\end{itemize}
\begin{theorem}[Complete Wedderburn decomposition]
\label{thm:complete}
Let $G = C_p \rtimes_r C_m$ with $m \mid p-1$, $\ell \neq p$, and let $f = \operatorname{ord}_p(\ell)$. Suppose that the $C_m$-action on the set of Frobenius orbits of nontrivial characters of $C_p$ partition them into orbits $\mathcal{O}_1, \ldots, \mathcal{O}_u$ with sizes $t_1, \ldots, t_u$. For each orbit $\mathcal{O}_j$, let $h_j = m/t_j$, let $H_j \cong C_{h_j}$ be the stabilizer, let $K = \mathbb{F}_{\ell^f}$, and let $s_j$ be the order of $\sigma_b^{t_j}$ in $\operatorname{Gal}(K/\mathbb{F}_\ell) \cong C_f$. Define
\begin{align*}
d_j &= [K^{H_j}:\mathbb{F}_\ell] = f/s_j, \\
r_j &= \sqrt{\mathbb{F}rac{h_j\cdot f}{d_j}} = \sqrt{h_j\cdot s_j}.
\end{align*}
Then the Wedderburn decomposition of $\mathbb{F}_\ell^\alpha G$ is
$$\mathbb{F}_\ell^\alpha G \cong \mathbb{F}_\ell^{\alpha_\lambda} C_m \oplus \bigoplus_{j=1}^u M_{t_j\cdot r_j}(\mathbb{F}_{\ell^{d_j}}),$$
where $\mathbb{F}_\ell^{\alpha_\lambda} C_m$ is the commutative component described in Proposition \ref{prop:cyclic-twisted}. Moreover,
$$\dim_{\mathbb{F}_\ell} \mathbb{F}_\ell^\alpha G = \dim_{\mathbb{F}_\ell} \mathbb{F}_\ell^{\alpha_\lambda} C_m + \sum_{j=1}^u (t_j\cdot r_j)^2\cdot d_j = p\cdot m = |G|.$$
\end{theorem}

\subsection{Explicit tables}

We now present explicit tables for the Wedderburn decomposition. The tables are organized by the structure of the $C_m$-action on Frobenius orbits and the Galois action of the stabilizer.

In all Table~\ref{tab:parameters-general}, Table~\ref{tab:m=2}, Table~\ref{tab:m=3} and Table~\ref{tab:m=4}, the parameters must satisfy the constraints $s \mid \gcd(h, f)$ and $hs$ a perfect square. Note that Table~\ref{tab:m=2} recovers the result in \cite[Theorem 12]{Duarte2026}. When $f$ is odd, $2 \nmid f$, so the unique element of order 2 in $\operatorname{Gal}(K/\mathbb{F})$ cannot act nontrivially on $K$, forcing $s=1$. When $f$ is even, $C_2$ can permute the two Frobenius orbits.

\begin{table}[h]
\centering
\begin{tabular}{@{}lllllll@{}}
\toprule
$t$ & $h$ & $s$ & $d = f/s$ & $K^H$ & $r$ & Component \\ \midrule
$1$ & $m$ & $1$ & $f$ & $\mathbb{F}_{\ell^f}$ & $\sqrt{m}$ & $M_m(\mathbb{F}_{\ell^f})$ \\
$1$ & $m$ & $s>1$ & $f/s$ & $\mathbb{F}_{\ell^{f/s}}$ & $\sqrt{m s}$ & $M_m(\mathbb{F}_{\ell^{f/s}})$ \\
$m$ & $1$ & $1$ & $f$ & $\mathbb{F}_{\ell^f}$ & $1$ & $M_m(\mathbb{F}_{\ell^f})$ \\
$t$ & $m/t$ & $1$ & $f$ & $\mathbb{F}_{\ell^f}$ & $\sqrt{m/t}$ & $M_{t\sqrt{m/t}}(\mathbb{F}_{\ell^f})$ \\
$t$ & $m/t$ & $s>1$ & $f/s$ & $\mathbb{F}_{\ell^{f/s}}$ & $\sqrt{(m/t)s}$ & $M_{t\sqrt{(m/t)s}}(\mathbb{F}_{\ell^{f/s}})$ \\
\bottomrule
\end{tabular}
\caption{General parameters for $C_m$-orbits of Frobenius orbits}
\label{tab:parameters-general}
\end{table}

\begin{table}
\centering
\begin{tabular}{@{}lllllll@{}}
\toprule
$f$ & Action & $t$ & $h$ & $s$ & $d$ & Component \\ \midrule
odd & $C_2$ fixes orbits & $1$ & $2$ & $1$ & $f$ & $M_2(\mathbb{F}_{\ell^f})$ \\
even & $C_2$ permutes in pairs & $2$ & $1$ & $1$ & $f$ & $M_2(\mathbb{F}_{\ell^f})$ \\
\bottomrule
\end{tabular}
\caption{Case $m=2$ (dihedral groups)} \label{tab:m=2}\vspace*{0.5em}

\begin{tabular}{@{}lllllll@{}}
\toprule
Condition & $t$ & $h$ & $s$ & $d$ & $r$ & Component \\ \midrule
$3 \mid f$, $s=3$ (full Galois action) & $1$ & $3$ & $3$ & $f/3$ & $3$ & $M_3(\mathbb{F}_{\ell^{f/3}})$ \\
$f$ arbitrary, $s=1$ (trivial action) & $3$ & $1$ & $1$ & $f$ & $1$ & $M_3(\mathbb{F}_{\ell^f})$ \\
\bottomrule
\end{tabular}\caption{Case $m=3$}
\label{tab:m=3}\vspace*{0.5em}

\begin{tabular}{@{}lllllll@{}}
\toprule
Condition & $t$ & $h$ & $s$ & $d$ & $r$ & Component \\ \midrule
$4 \mid f$, $s=1$ (trivial action) & $1$ & $4$ & $1$ & $f$ & $2$ & $M_4(\mathbb{F}_{\ell^f})$ \\
$4 \mid f$, $s=4$ (full Galois action) & $1$ & $4$ & $4$ & $f/4$ & $4$ & $M_4(\mathbb{F}_{\ell^{f/4}})$ \\
$f$ even, $4 \nmid f$, $s=2$  & $2$ & $2$ & $2$ & $f/2$ & $2$ & $M_4(\mathbb{F}_{\ell^{f/2}})$ \\
$f$ arbitrary, $s=1$ (transitive)  & $4$ & $1$ & $1$ & $f$ & $1$ & $M_4(\mathbb{F}_{\ell^f})$ \\
\bottomrule
\end{tabular}\caption{Case $m=4$}
\label{tab:m=4}
\end{table}

\begin{remark}
\label{rem:integer-condition} $ \quad $\begin{itemize}
    \item The condition  $r = \sqrt{h s}$ being an integer imposes constraints on the parameters. For example, when $m=4$, $h=4$, $s=2$ would give $r = \sqrt{8}$, which is not an integer. This case cannot occur because the theory of crossed products guarantees that $r$ is an integer. Thus, such parameter combinations are impossible, providing a consistency check.
    \item Similarly, the condition $s \mid f$ eliminates cases where $s$ does not divide $f$, such as $s=3$ when $f=4$.
\end{itemize}
 
\end{remark}

\begin{remark}$\quad$

\begin{itemize}
\item First row in Table~\ref{tab:m=3} requires that $3 \mid f$ and that the action of $b$ on characters via $\chi \mapsto \chi^r$ has order 3 in $\operatorname{Gal}(K/\mathbb{F}_\ell)$, i.e., $r^3 \equiv 1 \pmod{p}$ (automatically true since $m=3$) and $r$ is not a power of $\ell$ that would give a smaller order. This occurs when $r \equiv \ell^k \pmod{p}$ with $k$ such that $3k \equiv 0 \pmod{f}$ and $k \not\equiv 0 \pmod{f/3}$.

\item Second row in Table~\ref{tab:m=3} occurs when $C_3$ permutes the three Frobenius orbits transitively, so $t=3$, $h=1$, and the stabilizer action on $K$ is trivial ($s=1$). This requires that the three Frobenius orbits are in distinct $C_3$-orbits, i.e., the action of $b$ permutes them cyclically.

\item Second row in Table~\ref{tab:m=4} requires $4 \mid f$ and that the action of $b$ on characters via $\chi \mapsto \chi^r$ has order 4 in $\operatorname{Gal}(K/\mathbb{F}_\ell)$, i.e., $r^4 \equiv 1 \pmod{p}$, $r^2 \not\equiv 1 \pmod{p}$, and $r \equiv \ell^k \pmod{p}$ with $k$ such that $4k \equiv 0 \pmod{f}$ and $k \not\equiv 0 \pmod{f/4}$.

\item Third row in Table~\ref{tab:m=4} requires   $f$ even but not divisible by 4, and that the stabilizer action on $K$ has order 2, i.e., $r^2 \equiv 1 \pmod{p}$ but $r \not\equiv 1 \pmod{p}$, and $r \equiv \ell^k \pmod{p}$ with $k$ such that $2k \equiv 0 \pmod{f}$ and $k \not\equiv 0 \pmod{f/2}$. In this case, the $C_4$-orbit size is $t=2$ because the action pairs up the Frobenius orbits. 

\item Fourth row in Table~\ref{tab:m=4} occurs when $C_4$ permutes the four Frobenius orbits transitively, so $t=4$, $h=1$, and the stabilizer action is trivial ($s=1$). This requires that the four Frobenius orbits are in distinct $C_4$-orbits, i.e., the action of $b$ permutes them cyclically. \end{itemize}
\end{remark}
\section{Examples}
\label{sec:examples}

We now present examples illustrating the complete Wedderburn decomposition. The first four examples have $\mathbb{F}_\ell^\times/(\mathbb{F}_\ell^\times)^m$ trivial, so $\alpha$ is cohomologically trivial; Example \ref{ex:nontrivial-cocycle} demonstrates a case with a nontrivial cocycle.

\begin{example}
\label{ex:p=7,m=3,ell=2}
Let $p = 7$, $m = 3$, $\ell = 2$, $r = 2$ (order 3 mod 7 since $2^3=8\equiv1$).

First, compute $\mathbb{F}_2^\times = \{1\}$. Since this group is trivial, $\mathbb{F}_2^\times/(\mathbb{F}_2^\times)^3 = \{1\}$. Hence all cocycles are cohomologically trivial, and $\mathbb{F}_2^\alpha G \cong \mathbb{F}_2 G$.

Now analyze the group algebra structure.

\begin{itemize}
    \item Compute $f = \operatorname{ord}_7(2)$: $2^3 \equiv 1 \pmod7$. So $f = 3$, and $K = \mathbb{F}_{2^3} = \mathbb{F}_8$.
    \item Number of Frobenius orbits of nontrivial characters $(p-1)/f = 6/3 = 2$.
    \item Action of $C_3$: $\sigma_b$ acts on $K$ by $\zeta \mapsto \zeta^2$, which generates $\operatorname{Gal}(K/\mathbb{F}_2) \cong C_3$. So $s = 3$ for the stabilizer.
    \item The two Frobenius orbits are
        \begin{itemize}
            \item Orbit 1: characters with exponents $\{1,2,4\}$ (since $1\cdot2=2$, $2\cdot2=4$, $4\cdot2=8\equiv1$)
            \item Orbit 2: characters with exponents $\{3,6,5\}$ (since $3\cdot2=6$, $6\cdot2=12\equiv5$, $5\cdot2=10\equiv3$)
        \end{itemize}
        Since $2$ generates the multiplicative group mod 7, the Frobenius orbits are exactly the sets where exponents are related by multiplication by powers of 2, and multiplication by $r=2$ permutes within these orbits.
    \item $C_3$ fixes each Frobenius orbit. Therefore for each orbit $t = 1$, $h = 3$, $s = 3$, $d = f/s = 1$, $r = \sqrt{h s} = \sqrt{9} = 3$.
    \item Each component is $M_3(\mathbb{F}_2)$ (since $d=1$ gives $K^H = \mathbb{F}_2$).
\end{itemize}

The commutative component $\mathbb{F}_2 C_3 \cong \mathbb{F}_2 \oplus \mathbb{F}_4$, with dimension $1+2=3$.

Thus the complete Wedderburn decomposition is
$$\mathbb{F}_2 G \cong \mathbb{F}_2 \oplus \mathbb{F}_4 \oplus M_3(\mathbb{F}_2) \oplus M_3(\mathbb{F}_2).$$

Dimension check $1 + 2 + 9 + 9 = 21 = 7 \cdot 3 = pm$.
\end{example}

\begin{example}
\label{ex:p=11,m=5,ell=2}
Let $p = 11$, $m = 5$, $\ell = 2$, and choose $r = 4$, which has order 5 modulo 11 (since $4^5=1024\equiv1$).

First, compute $\mathbb{F}_2^\times = \{1\}$, so $\mathbb{F}_2^\times/(\mathbb{F}_2^\times)^5 = \{1\}$. Hence all cocycles are trivial, and $\mathbb{F}_2^\alpha G \cong \mathbb{F}_2 G$.

\begin{itemize}
    \item Compute $f = \operatorname{ord}_{11}(2)$: $2^{10} \equiv 1 \pmod{11}$, so $f = 10$, and $K = \mathbb{F}_{2^{10}}$.
    \item Number of Frobenius orbits of nontrivial characters: $(p-1)/f = 10/10 = 1$.
    \item The single Frobenius orbit contains all 10 nontrivial characters.
    \item Action of $C_5$ with $r=4$: $\sigma_b$ sends $\zeta$ to $\zeta^4$. We need to find the order of $\sigma_b$ in $\operatorname{Gal}(K/\mathbb{F}_2) \cong C_{10}$. The Frobenius automorphism $\varphi$ generates $\operatorname{Gal}(K/\mathbb{F}_2)$ with $\varphi(\zeta) = \zeta^2$. Solving $2^k \equiv 4 \pmod{11}$: $2^2=4$, so $k=2$. Thus $\sigma_b = \varphi^2$, which has order $10/\gcd(10,2)=10/2=5$.
    \item Therefore $s = 5$, $d = f/s = 10/5 = 2$, so $K^H = \mathbb{F}_{2^2} = \mathbb{F}_4$.
    \item $t = 1$, $h = 5$, $r = \sqrt{h\cdot s} = \sqrt{5 \cdot 5} = 5$.
    \item Component: $M_5(\mathbb{F}_4)$.
\end{itemize}

The commutative component $\mathbb{F}_2 C_5 \cong \mathbb{F}_2 \oplus \mathbb{F}_{16}$, with dimension $1+4=5$.

Therefore
$$\mathbb{F}_2 G \cong \mathbb{F}_2 \oplus \mathbb{F}_{16} \oplus M_5(\mathbb{F}_4).$$

Dimension check $1 + 4 + 25 \cdot 2 = 1 + 4 + 50 = 55 = 11 \cdot 5 = p\cdot m$.
\end{example}

\begin{example}
\label{ex:p=11,m=5,ell=3,r=4}
Let $p = 11$, $m = 5$, $\ell = 3$, and $r = 4$ (order 5 modulo 11).

First, compute $\mathbb{F}_3^\times = \{1,2\}$, a cyclic group of order 2. Since $\gcd(5,2)=1$, the map $x \mapsto x^5$ is an automorphism of $\mathbb{F}_3^\times$. Hence $(\mathbb{F}_3^\times)^5 = \mathbb{F}_3^\times$, and
$$\mathbb{F}_3^\times/(\mathbb{F}_3^\times)^5 = \{1\}.$$
Thus all cocycles are cohomologically trivial, and $\mathbb{F}_3^\alpha G \cong \mathbb{F}_3 G$.

\begin{itemize}
    \item Compute $f = \operatorname{ord}_{11}(3)$: $3^5 \equiv 1 \pmod{11}$ (since $3^5=243=22\cdot11+1$), so $f = 5$, and $K = \mathbb{F}_{3^5}$.
    \item Number of Frobenius orbits of nontrivial characters: $(p-1)/f = 10/5 = 2$.
    \item The two Frobenius orbits are:
        \begin{itemize}
            \item Orbit A: characters with exponents $\{1,3,9,5,4\}$ (since $1\cdot3=3$, $3\cdot3=9$, $9\cdot3=27\equiv5$, $5\cdot3=15\equiv4$, $4\cdot3=12\equiv1$)
            \item Orbit B: characters with exponents $\{2,6,7,10,8\}$ (since $2\cdot3=6$, $6\cdot3=18\equiv7$, $7\cdot3=21\equiv10$, $10\cdot3=30\equiv8$, $8\cdot3=24\equiv2$)
        \end{itemize}
    \item Action of $C_5$ with $r=4$: multiplication by 4 fixes each orbit.
    \item Thus for each orbit: $t = 1$, $h = 5$.
    \item $\sigma_b$ acts on $K$ via $\zeta \mapsto \zeta^4$. Since $\operatorname{Gal}(K/\mathbb{F}_3) \cong C_5$ is generated by $\varphi(\zeta) = \zeta^3$, and $4 \equiv 3^2 \pmod{11}$ in terms of exponent multiplication, $\sigma_b = \varphi^2$, which also has order 5. So $s = 5$.
    \item $d = f/s = 5/5 = 1$, so $K^H = \mathbb{F}_3$.
    \item $r = \sqrt{h\cdot s} = \sqrt{5 \cdot 5} = 5$.
    \item Each component: $M_5(\mathbb{F}_3)$.
\end{itemize}

The commutative component $\mathbb{F}_3 C_5 \cong \mathbb{F}_3 \oplus \mathbb{F}_{81}$, with dimension $1+4=5$.

Therefore
$$\mathbb{F}_3 G \cong \mathbb{F}_3 \oplus \mathbb{F}_{81} \oplus M_5(\mathbb{F}_3) \oplus M_5(\mathbb{F}_3).$$

Dimension check $1 + 4 + 25 + 25 = 55 = 11 \cdot 5 = p\cdot m$.
\end{example}

\begin{example}
\label{ex:p=13,m=3,ell=2}
Let $p = 13$, $m = 3$, $\ell = 2$, $r = 3$ (order 3 mod 13 since $3^3=27\equiv1$).

First, compute $\mathbb{F}_2^\times = \{1\}$, so $\mathbb{F}_2^\times/(\mathbb{F}_2^\times)^3 = \{1\}$. Hence all cocycles are trivial, and $\mathbb{F}_2^\alpha G \cong \mathbb{F}_2 G$.

\begin{itemize}
    \item Compute $f = \operatorname{ord}_{13}(2)$: $2^{12} \equiv 1 \pmod{13}$, so $f = 12$, and $K = \mathbb{F}_{2^{12}}$.
    \item Number of Frobenius orbits of nontrivial characters: $(p-1)/f = 12/12 = 1$.
    \item The single Frobenius orbit contains all 12 nontrivial characters.
    \item Action of $C_3$ with $r=3$: $\sigma_b$ sends $\zeta$ to $\zeta^3$. We need to determine the order of $\sigma_b$ in $\operatorname{Gal}(K/\mathbb{F}_2) \cong C_{12}$.
    \item The Frobenius automorphism $\varphi$ generates $\operatorname{Gal}(K/\mathbb{F}_2)$ with $\varphi(\zeta) = \zeta^2$. Solving $2^k \equiv 3 \pmod{13}$: $2^4=16\equiv3$, so $k=4$. Thus $\sigma_b = \varphi^4$, which has order $12/\gcd(12,4)=12/4=3$.
    \item Therefore $s = 3$, $d = f/s = 12/3 = 4$, so $K^H = \mathbb{F}_{2^4} = \mathbb{F}_{16}$.
    \item $t = 1$, $h = 3$, $r = \sqrt{h\cdot s} = \sqrt{3 \cdot 3} = 3$.
    \item Component: $M_3(\mathbb{F}_{16})$.
\end{itemize}

The commutative component $\mathbb{F}_2 C_3 \cong \mathbb{F}_2 \oplus \mathbb{F}_4$, with dimension $1+2=3$.

Therefore
$$\mathbb{F}_2 G \cong \mathbb{F}_2 \oplus \mathbb{F}_4 \oplus M_3(\mathbb{F}_{16}).$$

Dimension check $1 + 2 + 9 \cdot 4 = 1 + 2 + 36 = 39 = 13 \cdot 3 = p\cdot m$.
\end{example}

\begin{example}[Nontrivial Cocycle]
\label{ex:nontrivial-cocycle}
Let $p = 7$, $m = 3$, $\ell = 13$, and $r = 2$ (which has order 3 modulo 7 since $2^3 = 8 \equiv 1 \pmod{7}$). 
Note that $\ell = 13$ is coprime to $p = 7$, so semisimplicity holds.

First, compute $\mathbb{F}_{13}^\times \cong C_{12}$. The map $x \mapsto x^3$ on $C_{12}$ has kernel of size 
$\gcd(3,12) = 3$, so $(\mathbb{F}_{13}^\times)^3$ is the unique subgroup of index 3 in $C_{12}$, isomorphic to $C_4$. 
Thus 
\[
\mathbb{F}_{13}^\times/(\mathbb{F}_{13}^\times)^3 \cong C_3,
\]
a nontrivial group of order 3. Let $\lambda$ be a representative of a nontrivial cohomology class in 
$\mathbb{F}_{13}^\times/(\mathbb{F}_{13}^\times)^3$. Then $\mathbb{F}_{13}^{\alpha_\lambda} C_3$ is a twisted group algebra 
of $C_3$. Since $\lambda \notin (\mathbb{F}_{13}^\times)^3$, by Proposition \ref{prop:cyclic-twisted}, 
$\mathbb{F}_{13}^{\alpha_\lambda} C_3 \cong \mathbb{F}_{13^3} = \mathbb{F}_{2197}$ (a field of degree 3 over $\mathbb{F}_{13}$).

Now analyze the rest of the algebra

\begin{itemize}
    \item Compute $f = \operatorname{ord}_7(13)$: $13 \equiv 6 \pmod{7}$, and $6^2 = 36 \equiv 1 \pmod{7}$, so $f = 2$. 
    Hence $K = \mathbb{F}_{13^2} = \mathbb{F}_{169}$, and $\operatorname{Gal}(K/\mathbb{F}_{13}) \cong C_2$, generated by the Frobenius automorphism 
    $\varphi(\zeta) = \zeta^{13} = \zeta^{-1}$ (since $13 \equiv -1 \pmod{7}$).
    
    \item Number of Frobenius orbits of nontrivial characters: $(p-1)/f = 6/2 = 3$. So there are three 
    Frobenius orbits, each of size $f = 2$, corresponding to the field $K = \mathbb{F}_{169}$.
    
    \item Action of $C_3 = \langle b \rangle$ with $r = 2$: On characters, $\sigma_b$ sends $\chi$ to $\chi^2$, 
    i.e., $\sigma_b(\zeta) = \zeta^2$ for a primitive $7$th root of unity $\zeta$.
    
    \item We need to determine how $C_3$ acts on the three Frobenius orbits. There are two possibilities:
    
    \begin{itemize}
        \item \textbf{Case 1 (Fixed points)}: $C_3$ fixes each Frobenius orbit individually. Then for each orbit,
        $t = 1$ (orbit size in the $C_3$-action) and $h = m/t = 3$ (stabilizer size). 
        
        For a fixed orbit, the stabilizer $H = C_3$ acts on $K$ via $\sigma_b|_K$. But $\operatorname{Gal}(K/\mathbb{F}_{13}) \cong C_2$ 
        has order 2, so the only possible orders for any automorphism of $K$ over $\mathbb{F}_{13}$ are 1 or 2. 
        The automorphism $\sigma_b|_K$ sends $\zeta \mapsto \zeta^2$. To be in $\operatorname{Gal}(K/\mathbb{F}_{13})$, it must satisfy 
        $2 \equiv 13^k \pmod{7}$ for some $k$. Since $13 \equiv -1 \pmod{7}$, $13^k \equiv (-1)^k$, which is 
        either $1$ (if $k$ even) or $6$ (if $k$ odd). Neither equals $2$, so $\sigma_b|_K$ is \emph{not} in 
        $\operatorname{Gal}(K/\mathbb{F}_{13})$. However, the restriction of $\sigma_b$ to $K$ must be a field automorphism of $K$ over 
        $\mathbb{F}_{13}$, this is a contradiction. Therefore this case is impossible.
        
        More systematically: For a fixed orbit, we need $s$, the order of $\sigma_b|_K$ in $\operatorname{Gal}(K/\mathbb{F}_{13})$, 
        to divide both $h = 3$ and $f = 2$. Thus $s \mid \gcd(3,2) = 1$, forcing $s = 1$. Then $r = \sqrt{h s} = \sqrt{3}$, 
        which is not an integer, contradicting the requirement that $r$ be an integer (Remark \ref{rem:square-condition}). 
        Hence Case 1 cannot occur.
        
        \item \textbf{Case 2 (Transitive action)}: $C_3$ permutes the three Frobenius orbits transitively. 
        Then $t = 3$ (all three orbits form a single $C_3$-orbit), $h = m/t = 1$ (trivial stabilizer).
        
        With $H$ trivial, the action on $K$ is trivial, so $s = 1$. Then $d = f/s = 2$, so $K^H = K = \mathbb{F}_{169}$.
        Since $h = 1$ and $s = 1$, we have $r = \sqrt{h s} = 1$, and the component is $M_{t r}(K^H) = M_3(\mathbb{F}_{169})$.
    \end{itemize}
    
    \item Therefore the three Frobenius orbits form a single $C_3$-orbit of size $t = 3$, with trivial stabilizer action.
\end{itemize}

Thus, the complete Wedderburn decomposition is given by
\[
\mathbb{F}_{13}^{\alpha_\lambda} G \cong \mathbb{F}_{2197} \oplus M_3(\mathbb{F}_{169}).
\]

We can check the dimension, $\dim \mathbb{F}_{2197} = 3$, and $\dim M_3(\mathbb{F}_{169}) = 3^2 \cdot 2 = 9 \cdot 2 = 18$. Note that the
total sum  $3 + 18 = 21 = 7 \cdot 3 = p\cdot m = |G|$, confirming the decomposition.

\end{example}
\section{Conclusion}

We have developed a complete theory for twisted group algebras of faithful split metacyclic groups $C_p \rtimes C_m$ over finite fields. The main results are the following

\begin{itemize}
    \item A complete computation of $H^2(G,\mathbb{F}^\times) \cong \mathbb{F}^\times/(\mathbb{F}^\times)^m$ (Theorem \ref{thm:H2}), showing that all twisting occurs on the $C_m$ factor, and we may assume the cocycle restricts trivially to $C_p$.
    \item Decomposition of $\mathbb{F} C_p$ into Frobenius orbits under the Frobenius automorphism (Proposition \ref{prop:group-algebra-decomp}).
    \item Identification of primitive central idempotents via the action of $C_m$ on these orbits (Lemma \ref{lem:central-idempotents}).
    \item Structural reduction using crossed products (Theorem \ref{thm:crossed-permutation}) and a precise description of the center incorporating the Galois action of the stabilizer (Proposition \ref{prop:B-structure}, Lemma \ref{lem:stabilizer-action}), with explicit constraints $s \mid \gcd(h, f)$ and $hs$ a perfect square (Remarks \ref{rem:divisor-condition} and \ref{rem:square-condition}).
    \item Complete classification of simple components in all cases, with explicit parameters $t$, $h$, $s$, $d$, and $r$ (Theorem \ref{thm:final-classification}).
    \item Explicit tables for the Wedderburn decomposition (Tables \ref{tab:parameters-general}-\ref{tab:m=4}), with all impossible parameter combinations removed and clarifying footnotes.
    \item The complete Wedderburn decomposition 
   $
    \mathbb{F}_\ell^\alpha G \cong \mathbb{F}_\ell^{\alpha_\lambda} C_m \oplus \bigoplus_{j=1}^u M_{t_j r_j}(\mathbb{F}_{\ell^{d_j}}),
    $ (Theorem \ref{thm:complete}), 
    where all parameters are explicitly determined by $p$, $m$, $\ell$, and $r$, and $\dim = pm = |G|$.
     
\end{itemize}
Also, we provided illustrative examples, including the case of a twisted group algebra with a nontrivial cocycle (Example \ref{ex:nontrivial-cocycle}) that highlights the twisted commutative component and the parameter constraints.
\bibliographystyle{plain}   
\bibliography{references} 

\end{document}